\documentclass[letterpaper, 10 pt, conference]{ieeeconf}  % Comment this line out if you need a4paper

\IEEEoverridecommandlockouts                              % This command is only needed if
\let\labelindent\relax

\usepackage{amsmath,amssymb,amsfonts,amsthm}
\usepackage{amsmath,environ}
\usepackage[scr=dutchcal]{mathalpha}
\usepackage{cite}
\usepackage{grffile}
\usepackage{multicol}
\usepackage{caption}
\usepackage{graphics}
 \usepackage{subcaption}
\usepackage{url}
\usepackage{enumitem}
\usepackage{tabularx}
\let\mathopfont=\mathrm
\newcommand{\trace}{\mathop{\mathopfont{trace}}}

\newcommand{\mcl}[1]{\mathcal{ #1}}
\newcommand{\mbf}[1]{\mathbf{ #1}}
\newcommand{\norm}[1]{\left\Vert #1\right\Vert}

\newcommand{\ip}[2]{\left\langle{#1},{#2}\right\rangle}

\newcommand{\bmat}[1]{\begin{bmatrix} #1\end{bmatrix}}

\newcommand{\R}{\mathbb{R}}
\newcommand{\C}{\mathbb{C}}

\newtheorem{thm}{Theorem}
\newtheorem{defn}[thm]{Definition}
\newtheorem{lem}[thm]{Lemma}

\newtheorem{cor}[thm]{Corollary}

\newcommand{\PI}{\pmb{\Pi}}
\newcommand{\pie}{\scalebox{0.9}[1.2]{$\mathit{\Pi}$}}

\newcommand{\fourpi}[4]{\hspace{0.5mm}\pie\hspace{-1.mm}\left[\footnotesize\begin{array}{c|c}
#1&#2\\\hline #3 & \{#4\}
\end{array}\right]}

\allowdisplaybreaks

\title{\LARGE \bf
$H_2$-Optimal Estimation of Linear Delayed and PDE Systems}

\author{Danilo Braghini$^{1}$, Sachin Shivakumar$^{2}$ and Matthew M. Peet$^{3}$% <-this % stops a space
\thanks{This work was supported by the National Science Foundation under grants No. 2337751 and 2429973.}% <-this % stops a space
\thanks{$^{1}$ Danilo Braghini\{{\tt\small dbraghini@asu.edu}\} and $^{3}$ Matthew M. Peet\{{\tt\small mpeet@asu.edu}\} are with the School for Engineering of Matter, Transport and Energy at Arizona State University. $^{2}$ Sachin Shivakumar\{{\tt\small sshivakumar@lanl.gov}\} is with Center for Nonlinear Studies and Applied Mathematics \& Plasma Physics Division of Los Alamos National Laboratory.
        }%
}
\begin{document}
\maketitle
\thispagestyle{empty}
\pagestyle{empty}
%%%%%%%%%%%%%%%%%%%%%%%%%%%%%%%%%%%%%%%%%%%%%%%%%%%%%%%%%%%%%%%%%%%%%%%%%%%%%%%%
\begin{abstract}
The $H_2$ norm is a commonly used performance metric in the design of estimators. However, $H_2$-optimal estimation of most PDEs is complicated by the lack of transfer function and state-space representations. To address this problem, we first re-characterize the $H_2$-norm in terms of a map from initial condition to output. We then leverage the Partial Integral Equation (PIE) state-space representation of systems of linear PDEs coupled with ODEs to recast this characterization of $H_2$-norm as a convex optimization problem defined in terms of Linear Partial Integral (LPI) inequalities. 
We then parameterize a class of PIE-based observers and solve the associated $H_2$-optimal estimation problem. The %optimal observer synthesis problem is then recast as an LPI, and the 
resulting observers are validated using numerical simulation.%\vspace{-2mm}

\end{abstract}
\section{Introduction}%\vspace{-1mm}

Partial Differential Equations (PDEs) are used to describe the evolution of processes whose states
are distributed over a spatial domain. Examples of such processes include fluid flow~\cite{chen2011h2}, vibroacoustics~\cite{ambrosio2012h2}, chemical reaction networks~\cite{unsleber2020exploration}, and time-delay systems~\cite{richard2003time}, where the corresponding distributed states
are velocity profile, displacement, species concentration, and history. For such systems, it is often desirable to be able to track the evolution of the system using sensor measurements -- either for the purpose of feedback control~\cite{smyshlyaev2004closed} 
or for monitoring and fault detection~\cite{huang2009fault}.%,li2016optimal}.

Unlike Ordinary Differential Equations (ODEs) and other such lumped-parameter systems, however, direct measurement of the system state of a PDE requires an uncountable number of sensors -- a practical impossibility. Consequently, there has been significant interest in the development of observers wherein, by tracking a finite set of measurements, we may infer real-time estimates of the entire distributed state. For ODEs, the problem of state estimation has been largely solved, with special cases
including the Luenberger observer, the Kalman filter, and Linear Matrix Inequalities (LMIs) for $H_{\infty}$-optimal observers and filters -- methods that can be applied to state estimation for any linear ODE with state-space representation.
However, for PDEs, the need to integrate boundary conditions and the distributed states precludes the existence of a convenient and universal state-space representation. This means that most efforts to design estimators for such systems are ad hoc -- requiring significant modification for even minor changes in the model~\cite{backstepping}. As a result, most approaches to the estimation of the PDE state entail a reduction of the PDE state to finite dimensions, either through early-lumping~\cite{mechhoud2014estimation} %, felici2014dynamic}
(by reducing the distributed states to finite dimensions), or late-lumping~\cite{benner2018numerical} %,moghadam2013boundary}
(which enforces synthesis conditions on a finite number of test functions).
%Early-lumping~\cite{mechhoud2014estimation, balas1978active, felici2014dynamic} entails the reduction of the PDE to a state-space ODE through  Galerkin projection, spatial discretization or modal decomposition. Late-lumping~\cite{schmidt2016reduced,benner2018numerical,moghadam2013boundary}, by contrast, formulates synthesis conditions using the distributed state but then enforces those conditions on a finite number of test functions. In both cases, neither the stability nor performance of the estimator can be proven unless the truncation error can be bounded through some auxiliary, ad-hoc process.

Recently, efforts have been made to synthesize observers for PDE systems without lumping through the use of a more convenient state-space representation of PDEs. This method integrates the PDE evolution equation with the boundary conditions by defining the state as the highest spatial derivative of the distributed state and parameterizing the evolution of this state by means of integral operators with polynomial kernels. This method has the advantage that such operators form an algebra, which can be represented using matrices and optimized using LMIs. The representation of a PDE using such operators is referred to as a Partial Integral Equation (PIE), and methods for the construction of PIE representations of a broad class of PDEs are well-established~\cite{peet2021partial,shivakumarGPDE,%jagt_2022CDC,
peet2021representation}.% Unlike other approaches for estimation, such as lumping, PIE representation allows the formulation of computationally tractable synthesis conditions with provable bounds for the stability and performance. %Moreover, since the estimator synthesis conditions for a PIE are universal and the PIE representation applies to a broad class of PDEs, there are no ad-hoc steps as in other traditional approaches, e.g., Backstepping.

Observer designs which minimize an $L_2$-gain bound for PDE systems that admit a PIE representation have previously been presented in~\cite{das_cdc2019,jagt2024h} and for time-delay systems in~\cite{wu2023h}. 
%These results parameterize the observer dynamics using PIEs and pose conditions for stability and performance bounds as the solution of a convex optimization problem expressed in terms of Partial Integral (PI) operator variables and Linear Partial Integral Inequalities (LPIs), which can be enforced using %a recently developed 
%Matlab toolboxes such as
%~\cite{manual}. These results %use a generalization of the KYP lemma to partial integral equations to
%ensure stability and bound the $L_2$-gain of a tracking error to disturbances such as sensor noise. 
The problem with minimization of $L_2$-gain, however, is that disturbances such as sensor noise are not typically characterized in terms of energy, but rather in terms of frequency content and power spectral density -- implying that the $H_2$ norm is a more suitable performance metric in the design of observers (e.g., LQG and Kalman filters).

The goal of this paper, then, is to formulate and solve the problem of $H_2$-optimal observer synthesis. Unlike $H_{\infty}$-optimal %observer synthesis,
estimation, wherein a proxy for $H_\infty$ performance is $L_2$-gain, the main technical difficulty for $H_2$-optimal estimation is the identification of a time-domain proxy for $H_2$ performance. To address this difficulty, we rely on an initial condition to output characterization of the $H_2$ metric as proposed in~\cite{danilo}. This allows us to extend classical LMIs for $H_2$-performance to LPI-type conditions to performance bounds on the error dynamics of the PIE-based observer.

This paper is structured as follows. First, Section~\ref{sec:preliminaries} defines PI operators, PIEs, and LPIs. Section~\ref{sec:problems} introduces a time-domain characterization of the $H_2$ norm and formulates $H_2$-optimal observer synthesis problem. Section~\ref{sec:h2norm} gives an LPI characterization of the $H_2$-norm of a PIE, and Section~\ref{sec:h2estimator} extends this result to give an LPI condition for computing $H_2$-optimal observer gains. Section~\ref{sec:reconstruction} gives a procedure to find observer gains from the LPI solution, and Section~\ref{sec:num} presents numerical examples for observer validation. 

% A procedure for reconstruction of the observer gains from the solution to the LPI is given in Section~\ref{sec:reconstruction} and the observers are tested using numerical simulation in Section~\ref{sec:num}.
\textbf{Notation}: $L_2^p[a,b]$ and $L_2^p[0,\infty]$ are the spaces of \textit{Lebesgue} square-integrable $\R^p$-valued functions.
% $W_{k}^n[a,b]$ denotes the Sobolev space \begin{align*}
% W_{k}^n[a,b] := \{\mbf x \in L_2^n[a,b] : \partial^i_s \mbf x\in L_2^n[a,b] ~\text{for all} ~i\le k\}
% \end{align*}
% where $\partial_s^j \mbf x$ denotes the partial derivative $\frac{\partial^j\mbf x}{\partial s^j}$. $W_{k}^n[a,b]$ is equipped with the standard Sobolev inner-product, denoted $\ip{\cdot}{\cdot}_{W_k}$. In this paper, Sobolev spaces are associated only with a compact spatial domain.
$\R L_2^{m,p}[a,b]$ denotes the Hilbert space $\R^{m}\times L_2^{p}[a,b]$. Occasionally, %we omit 
the domain  %and simply write $L_2^p$ or $\R L_2^{m,p}$
is omitted when clear from context.%\vspace{- 2 mm}
%We use the bold font (e.g., $\mbf x$) to indicate scalar or vector-valued functions of a spatial variable. 
%For Hilbert spaces $X,Y$, $\mcl L(X,Y)$ denotes the set of bounded linear operators from $X$ to $Y$ with $\mcl L(X):=\mcl L(X,X)$.\vspace{- 1 mm}

%We use the calligraphic font (e.g. $\mcl{A}$) to represent such bounded linear operators.\vspace{-1.5mm}
%
%For any $\mcl A\in \mcl L(Y,X)$, $\mcl A^*$ denotes the adjoint operator satisfying $\ip{\mbf x}{\mcl A\mbf y}_{X} = \ip{\mcl A^*\mbf x}{\mbf y}_{Y}$ for all $\mbf x \in X, \mbf y \in Y$. $\mcl A$ is self adjoint if $\mcl A=\mcl A^*$ and for such self adjoint operators, $\mcl A \succcurlyeq 0$ means $\ip{\mbf x}{\mcl A \mbf x}\ge 0$ for all $\mbf x\in X$.% $\mcl A\succ 0$ means and positive definiteness and semidefiniteness of self-adjoint PI operators. The set of 4-PI operators is denoted $\PI_4$, and the cone of positive 4-PI operators, $\PI_4^+$.
%
%In this section, we introduce the notation used in the paper and we provide the basis the
\section{State Space and Convex Optimization: PIs, PIEs, and LPIs}\label{sec:preliminaries}%\vspace{- 1 mm}

In this section, we introduce the algebra of 
Partial Integral (PI) operators, the class of systems modeled using Partial Integral Equations (PIEs), and the class of convex optimization problems defined in terms of Linear PI  (LPI) Inequality constraints.\vspace{-2mm}
%
% and
% Partial Integral (PI) operators are a class of operators that generalize the algebraic structure of matrices to infinite-dimensional spaces. Partial Integral Equations (PIEs) are a set of integro-differential equations, parameterized by PI operators. PIEs provide a state-space representation of a large class of linear time-invariant systems modeled by PDEs; it do so by removing continuity and boundary constraints associated with the PDE representation. Finally, the LPIs generalize the convex optimization problems of LMIs since its feasable spaces are the convex cone of positive PI operators.
%
%These three ingredients are combined in this section, providing the basis for the PIE framework, used in the control problems addressed by this work. The usefulness of PIE framework is illustrated in the last section, where state-estimators are synthesized for PDEs with provable bounds and without the necessity of tailored discretization schemes.

\subsection{The Algebra of Partial Integral Operators}

We begin by defining the algebra of partial integral operators, which will be used to parameterize PIEs in Subsec.~\ref{subsec:PIE}.\vspace{-1mm} %After properly defining these operators, we show how they expand the computationally tractable parametrization in finite-dimensional systems, given by matrices, to infinite-dimensional systems; this motivates their use in PDE systems coupled with ODEs, where the state has both finite and infinite parts.

\begin{defn}\label{def:4PI}
Given a matrix $P$ and polynomials $Q_1,Q_2,R_0,R_1$, a \textbf{4-PI operator} $\mcl P= \fourpi{P}{Q_1}{Q_2}{R_i} \in \mcl L(\R L_2^{m_1,n_1},\R L_2^{m_2,n_2})$ is such that
\begin{align*}
&\left(\mcl P \bmat{x\\\mbf x} \right)(s) := \bmat{Px + \int_{a}^{b}Q_1(\theta)\mbf x(\theta)d\theta\\Q_2(s)x+ \mcl R\mbf{x} (s)},\;  \text{where}\\
&\left(\mcl R\mbf x\right)\hspace{-.25mm}(s)\hspace{-1mm}= \hspace{-1mm}R_0(s) \mbf x(s) +\hspace{-1.5mm}\int\limits_{a}^s  \hspace{-1.5mm}R_1(s,\theta)\mbf x(\theta)d \theta+\hspace{-1.5mm}\int\limits_s^b \hspace{-1.5mm}R_2(s,\theta)\mbf x(\theta)d \theta.
\end{align*}
\end{defn}
We refer to $\PI_4\subset \mcl L(\R L_2^{m_1,n_1},\R L_2^{m_2,n_2})$ as the set of 4-PI operators. If $m_1=m_2$ and $n_1=n_2$, this set of PI operators is closed under composition, addition, and adjoint; explicit formulae for these operations can be obtained in terms of the polynomial matrices used to parameterize them~\cite{shivakumarGPDE}.

As in Def.~\ref{def:4PI}, the notation $\fourpi{P}{Q_1}{Q_2}{R_i}$ is used to indicate the 4-PI operator associated with the matrix $P$ and polynomial parameters $Q_i$, $R_j$. The associated dimensions ($m_1,n_1,m_2,n_2$) are inherited from the dimensions of the constant matrix $P\in \R^{m_2 \times m_1}$ and polynomial matrices $Q_1(s) \in \R^{m_2 \times n_1}$, $Q_2(s) \in \R^{n_2 \times m_1}$, and $R_0(s),R_1(s,\theta), R_2(s,\theta) \in \R^{n_2 \times n_1}$. In the case where a dimension is zero, we use $\emptyset$ in place of the associated parameter with zero dimension.\vspace{-2mm} %For example, the particular case of $n_1=n_2=0$ makes
%\[
%P=\fourpi{P}{\emptyset}{\emptyset}{\emptyset} : \R^{m_1} \to \R^{m_2},
%\]
%a matrix.
%On the other hand, by taking $m=n=0$, the resultant operator has the form
%\[
%\fourpi{\emptyset}{\emptyset}{\emptyset}{R_i} :  L_2^{n_1} \to L_2^{n_2}.
%\]
%which is a particular class of 4-PI operators, called 3-PI, used in systems of PDEs where the state has only an infinite-dimensional component. It was shown in~\cite{peet2021partial} that 3-PI operators form a $*$-algebra, and thus, the set is endowed with algebraic operations used in matrix algebra: addition, composition, adjoint, and inverse.
%
%Other particular case used to represent finite dimensional inputs to state behavior is $n_1=0$, giving
%\[
%\fourpi{P}{\emptyset}{Q_2}{\emptyset} :  \R^{m_1} \to  \R L^{m_2,n_2}.
%\]
%
%Finally, to represent state to finite dimensional outputs, we use
%\[
%\fourpi{P}{Q_1}{\emptyset}{\emptyset} : \R L^{m_1,n_1} \to  \R^{m_2}.
%\]
%
%This generalization, together with the fact that polynomial matrices are easily stored in computers, allows the numerical results that have been presented so far using the PIEs framework: an algebraic mimic of established results in convex optimization of ODE systems in terms of PI operators and search for optimal solutions using the developed computational toolbox.

\subsection{Partial Integral Equations}\label{subsec:PIE}\vspace{-1 mm}

It has been shown in, e.g.~\cite{shivakumarGPDE}, that a large class of PDE coupled with ODEs, with sensed and regulated outputs $y(t) \in \R^{n_y}, z(t)\in \R^{n_z}$, and in-domain disturbances, $w(t)\in\R^{n_w}$, may be equivalently represented using a partial integral equation (PIE) of the form
%The problem covered by this paper is the state-estimation of PIE systems with the structure:
\begin{align}\label{PIE}
\partial_t(\mcl T \mbf x(t))&= \mcl A  \mbf x(t)+\mcl B_1 w(t), \qquad   \mbf x(0)\in \R L_2,\notag\\
z(t)&=\mcl C_1  \mbf x(t),\qquad y(t)=\mcl C_2  \mbf x(t) +  D_{21} w(t),
\end{align}
where %the parameters 
$\mcl T$, $\mcl A$, $\mcl B_1$, $\mcl C_1$, $\mcl C_2\in \PI_4$, $D_{21}\in \R^{n_y \times n_w}$. %, are all 4-PI operators.
The solution of the PIE, $\mbf x(t) \in \R L^{m,n}_2[a,b]$, yields a solution to the PDE as $\mcl T\mbf x(t)$. The PIE state, $\mbf x(t)$, combines the ODE state with a spatial derivative of the PDE state and admits no boundary conditions or continuity constraints.

The solution of this class of PIE is formally defined as follows, where $x \in L_{2e}^p[0,\infty)$ means $x(t)\in \R^p$ and $\int_0^T \norm{x(t)}^2dt$ is finite for all $T \geq 0$.\vspace{-1mm}

 \begin{defn}[PIE solution]\label{def:piesol}
        Given $\mcl T$, $\mcl A$, $\mcl B_1$, $\mcl C_1$, $\mcl C_2\in \PI_4$, $D_{21}\in \R^{n_y \times n_w}$, we say $\{\mbf x, z, y\}$ is a solution to the PIE system for given initial condition $\mbf x(0)\in \R L_2^{m,n}[a,b]$ and input $w\in L_{2e}^{n_w}[0,\infty)$, if $\mcl T\mbf x(t)$ is \textit{Fr\'echet} differentiable for all $t\in [0,\infty)$, and if $\mbf x(t) \in \R L_2^{m,n}[a,b]$, $z\in L_{2e}^{n_z}[0,\infty)$, and $y\in L_{2e}^{n_y}[0,\infty)$ satisfy Eq.~\eqref{PIE} for all $t\in [0,\infty)$.
 \end{defn}\vspace{-4mm}

\subsection{Linear PI Operator Inequalities}\label{subsec:LPIs}\vspace{-1mm}

As described in Subsec.~\ref{subsec:PIE}, a large class of PDEs coupled with ODEs admit a PIE representation parametrized by 4-PI operators of the form given in Def.~\ref{def:4PI}. Later, % in
Sec.~\ref{sec:h2estimator} %we
 shows that the problem of $H_2$-optimal estimator synthesis for PIEs can be reformulated as an optimization problem whose decision variables are 4-PI operators and have PI-operator valued affine-in-variable inequality constraints -- a form of convex optimization problem %we 
 defined as a Linear PI Inequality (LPI). To illustrate how LPIs may be solved, let us consider the following example from~\cite{danilo}. 

The LPI formulation of $H_2$-norm of a PIE, proposed in \cite{danilo}, involves a constraint of the form $\mcl Q:=-(\mcl A^{\ast} \mcl P \mcl T + \mcl T^{\ast} \mcl P \mcl A + \mcl C_1^{\ast} \mcl C_1) \succcurlyeq 0$ with $\mcl P\succcurlyeq 0$ where, if $\mcl P,\mcl A,\mcl T,\mcl C_1\in \PI_4$, then $\mcl Q \in \PI_4$. %Conditions of this This constraint is a linear PI-valued inequality whose parameters are matrix-valued polynomials, and thus, this optimization problem is an LPI optimization problem. %Optimization problems in this form may be solved by using positive matrices to parameterize the coefficients of the polynomials that define the PI operator. 
%To solve inequalities of this type, one must find $\mcl P\succ 0$ such that $\mcl X\preccurlyeq 0$. 
To verify feasibility of such conditions, we enforce positivity of the variables $\mcl P,\mcl Q$ as $\mcl P = \mcl Z^* P\mcl Z$ and $\mcl Q = \mcl Z^*Q\mcl Z$ where $\mcl Z$ is a fixed basis of PI-operators and $Q, P\succcurlyeq 0$ are %positive
matrix variables. Then we may enforce the equality %constraint
%\[
$\mcl A^*\mcl Z^* P\mcl Z\mcl T+\mcl T^*\mcl Z^*P\mcl Z\mcl A+ \mcl C_1^*\mcl C_1 = -\mcl Z^*Q\mcl Z,$
%\]
which is interpreted in terms of equality constraints on the coefficients of the polynomials which define $\mcl Z,\mcl A,\mcl T,{\mcl C}_1$. %Since two polynomials are equal if and only if their coefficients are equal, one can use Semidefinite Programming (SDP) to solve for these positive matrices that are subject to some equality constraints.
%Specifically, the existence of such a factorization assumes that a square root of $\mcl X$ exists, such that $\mcl X^{\frac{1}{2}} \in \PI_4$.% depending on $\mcl Z_d$ and $U$. 

The steps involved in the above LPI solution procedure, namely, parsing the LPIs, parameterizing decision variables using matrices, extracting the SDP constraints, and retrieving the operators from the solution of SDP, are automated in PIETOOLS Matlab toolbox~\cite{manual}.\vspace{-2mm}
%by using positive matrices to parameterize positive PI operators, as described in~\cite{peet2021partial}, and automated by the PIETOOLS Matlab toolbox~\cite{manual}..%For convenience, we denote the set of positive 4-PI operators that admit such a parameterization as $\PI_4^+$.
%Thus, affine inequality constraints of the form, e.g.
%\[
%\mcl A^{\ast} \mcl P \mcl T + \mcl T^{\ast} \mcl P \mcl A \preccurlyeq -\epsilon I,
%\]
%can be enforced by using the equality constraint $\mcl A^{\ast} \mcl P \mcl T + \mcl T^{\ast} \mcl P \mcl A+\epsilon I=-\mcl Q$ where $\mcl Q \in \PI_4^+$ and the equality constraint is enforced by equating the coefficients of the polynomials which define $\mcl Q\in\PI_4$ and $\mcl A^{\ast} \mcl P \mcl T + \mcl T^{\ast} \mcl P \mcl A+\epsilon I \in \PI_4$. 
%Software for the numerical solution of LPIs can be found in~\cite{manual}.

\section{Problem Formulation}\label{sec:problems}\vspace{-1 mm}

%In 
This section %, we 
introduces a suitable time-domain characterization of the $H_2$ norm %and uses this characterization
used to define the problems of $H_2$ norm bounding and $H_2$-optimal estimation for systems that admit a PIE representation.\vspace{-2mm}

\subsection{The $H_2$ norm of a PIE}\vspace{-1mm}

For this subsection, %we restrict our consideration to 
consider the characterization of the $H_2$ norm of a system represented by a PIE of the form
\begin{align}
\partial_t (\mcl T \mbf x(t))&= \mcl A \mbf x(t)+\mcl B_1 w(t),\, z(t)=\mcl C_1 \mbf x(t), \label{PIE2}
%\qquad y(t)=\mcl C_2  \mbf x(t) + \mcl D_{21} w(t),
\end{align}
with $\mcl T\mbf x(0) =0$, where $\mbf x(t) \in \R L_2^{m,n}[a,b]$ is the state, $w(t) \in \R^{n_w}$ is a disturbance, and $z(t) \in \R^{nz}$ is the output. Specifically, we define the $H_2$ norm of this system as the $L_2$-gain of initial condition to output of an auxiliary system with no disturbance. While non-standard, we will see that this characterization of $H_2$ performance is equivalent in a certain sense to the standard definition of $H_2$ norm.\vspace{-2 mm}

\begin{defn}\label{def:h2norm}
Consider solutions of the auxiliary PIE
%We say that System~\eqref{PIE2} has $H_2$ norm $\gamma$ if, for any $x_0 \in \R^{n_u}$ with $\norm{x_0}\le 1$ and $\mbf x, z$ which satisfy
\begin{align}
\partial_t (\mcl T \mbf x(t))&= \mcl A \mbf x(t), \, %\notag \\
z(t)=\mcl C_1 \mbf x(t), %\qquad
\, \mcl T \mbf{x}(0) =\mcl{B}_1 x_0,\label{eqn:PIEprime}
\end{align}
where $x_0 \in \R^{n_w}$. We define the $H_2$ norm of System~\eqref{PIE2} (denoted $G(\mcl T, \mcl A,\mcl B_1,\mcl C_1)$) as
\[
 \norm{G(\mcl T, \mcl A,\mcl B_1,\mcl C_1)}_{H_2}:=\sup_{\substack{z,\mbf x\, \text{satisfy~\eqref{eqn:PIEprime}}\\ \norm{x_0}_2=1}} \norm{z}_{L_2}.
\]
%we have that $\norm{z}_{L_2}\le \gamma$.
\end{defn}

To see the relationship between %the definition of $H_2$ norm in 
Def.~\ref{def:h2norm} and the standard definition of the $H_2$ norm, consider an ODE of the form
\begin{equation}
\bmat{\dot x(t)\\z(t)}=\bmat{A&B\\C& 0}\bmat{x(t)\\w(t)},\qquad  \forall t \in [0, \infty).\label{ODE}
\end{equation}
%Then if $A$ is Hurwitz, and we define the transfer function as $\hat G(s)=C(sI-A)^{-1}B$, then the $H_2$ norm is
%\begin{align*} \Vert{\hat G}\Vert^2_{H_2}&=\frac{1}{2\pi} \int_{-\infty}^{\infty} \trace\left(G^*(i \omega) G(i \omega) d\omega \right)\\
%=&\trace \left(B_1^T \int_{0}^{\infty} e^{A^T\tau} C_1^T C_1 e^{A\tau} d\tau B_1\right).
%\end{align*}
%where we have used the inverse Laplace transform to obtain the time-domain characterization~\cite{dullerud2013course}.
\vspace{-3 mm}

\begin{cor}\label{thm:H2defIC}
Suppose $A$ is Hurwitz and $\hat G(s)=C(sI-~A)^{-1}B$. %with $B \in \R^{n_x \times n_w}$.
Consider solutions of the auxiliary ODE
\begin{align}
\dot x(t)&= A  x(t), \quad%\notag \\
z(t)= C x(t), \quad %\qquad 
x(0) =B x_0.\label{ODE2}
\end{align}
Then
\[
\sup_{\substack{z, x\, \text{satisfies~\eqref{ODE2}}\\ \norm{x_0}_2=1}} \norm{z}_{L_2} \le \norm{\hat G}_{H_2}\le \sqrt{n_w} \sup_{\substack{z, x \, \text{satisfies~\eqref{ODE2}}\\ \norm{x_0}_2=1}} \norm{z}_{L_2}.
\]
\end{cor}\vspace{-2 mm}

\begin{proof}
	Suppose  $\{x,z\}$ satisfy~\eqref{ODE2} with initial condition $x(0)=B x_0$. Then $x(t) =e^{At}Bx_0$ and hence if $\norm{x_0}_2=1$, we have
	\begin{align*}
	\norm{z}_{L_2}^2&=\int_{0}^{\infty} x_0^T B^T e^{A^T\tau} C^T C e^{A\tau}Bx_0 d\tau\\
	%&\le \bar \sigma\left(\int_{0}^{\infty} B^T e^{A^T\tau} C^T C e^{A\tau}B d\tau\right)\\
	&\le \trace \left(\int_{0}^{\infty} B^T e^{A^T\tau} C^T C e^{A\tau}B d\tau\right)=\norm{\hat G}^2_{H_2}.
	\end{align*}
Furthermore,
	\begin{align*}
	\hspace{-1 mm}\frac{1}{n_w}\norm{\hat G}^2_{H_2} 
	\le \bar \sigma\left(\int_{0}^{\infty} B^T e^{A^T\tau} C^T C e^{A\tau}B d\tau\right)%\\
	%&= n_w \sup_{\norm{x_0}=1} \int_{0}^{\infty} x_0^T B^T e^{A^T\tau} C^T C e^{A\tau}Bx_0 d\tau\\
\hspace{-1 mm}= \hspace{-2 mm} \sup_{\norm{x_0}_2=1} \hspace{-2 mm}\norm{z}_{L_2}^2.
	\end{align*}
\end{proof}\vspace{-3mm}

Clearly, if the PIE has a single input, the proposed definition of $H_2$ norm coincides with the typical definition. In the case of multiple inputs, the time-domain characterization of $H_2$ norm would coincide with the %an 
alternative definition %of $H_2$ norm given by
\[
\Vert{\hat{G}}\Vert_{H_2}^2=\frac{1}{2\pi} \int_{-\infty}^{\infty} \bar{\sigma}\left(\hat G^*(\imath\omega) \hat G(\imath\omega) d\omega \right).
\]

Having defined the $H_2$-norm, we proceed to formulate the $H_2$-optimal estimator synthesis problem.\vspace{-2 mm}% Our approach derives a convex problem; optimal solutions can be ultimately searched by usual semi-definite programming solvers.

\subsection{$H_2$-Optimal Estimators}\label{subsec:problem_est}\vspace{-1 mm}

Our goal is to design observers for the class of coupled ODE-PDE systems that admit a PIE representation as%of the form
\begin{align}
\partial_t(\mcl T \mbf x(t))&= \mcl A  \mbf x(t)+\mcl B_1 w(t), \qquad  \mcl T\mbf x(0) =0,\notag\\
z(t)&=\mcl C_1  \mbf x(t),\qquad y(t)=\mcl C_2  \mbf x(t) + D_{21} w(t),\label{PIE_repeat}
\end{align}
where recall the state of the original PDE is obtained from the solution of the PIE as $\mcl T \mbf x(t)$. The signal $y(t)$ contains measurements of the PDE, and $z(t)$ represents those parts of the state by which we will measure the performance of our estimator. Our estimator dynamics are then assumed to have the Luenberger observer structure
\begin{align}
 \partial_t\left(\mcl T \tilde{\mbf x}(t) \right)&=\mcl A \tilde{\mbf x}(t) +\mcl L \left(\mcl C_2 \tilde{\mbf x}(t)- y(t)\right),\, \mcl T\tilde{\mbf x}(0)=  0,\label{PIEobserver2}
\end{align}
which mirror the dynamics of the observed system, but without the disturbance $w(t)$, which is unknown. The term, $\mcl C_2 \tilde{\mbf x}(t)- y(t)$, reflects the difference between the predicted and measured output from the PDE. This term is weighted by the observer gain, $\mcl L:\R^{n_y} \rightarrow \R L_2^{m,n}$, which is taken to be a PI operator. By combining the observer in Eq.~\eqref{PIEobserver2} with the measured output of a PDE, real-time estimates of the PDE state can be obtained as $\mcl T \tilde{\mbf x}(t)$ and used in conjunction with state-feedback controllers or fault detection algorithms.

The $H_2$-optimal estimation problem, then, is to choose $\mcl L$ which minimizes the $H_2$-norm of the map from disturbance $w$ to error in the regulated output, which we define as $e_z(t)=\mcl C_1\tilde{\mbf x}(t)-z(t)$. This map can likewise be represented as a PIE with state $\mbf e(t)=\tilde{\mbf x} (t)-\mbf x(t)$, where $\tilde{\mbf x}$ satisfies Eq.~\eqref{PIEobserver2} and $\mbf x$ satisfies Eq.~\eqref{PIE_repeat} so that
 \begin{align}\label{eq:PIEerror}
\partial_t\left(\mcl T \mbf e(t) \right) &= (\mcl A+\mcl L\mcl C_2) \mbf e(t)-\left(\mcl B_1+\mcl L D_{21} \right)w(t),\notag\\
e_z(t) &= \mcl C_1 \mbf e(t), \qquad \mcl T \mbf e(0)= 0.
\end{align}
We see that Eq.~\eqref{eq:PIEerror} is of the form in Eq.~\eqref{PIE2} with $\mcl A \mapsto \mcl A+\mcl L\mcl C_2$, $\mcl B_1 \mapsto -\left(\mcl B_1+\mcl L  D_{21} \right)$ and $\mcl C_1 \mapsto \mcl C_1$. %Thus 
%We can formulate 
Then, the $H_2$-optimal synthesis problem can be formulated as %using %the auxiliary PIE from 
%as Def.~\ref{def:h2norm}
%  \begin{align}
% \partial_t\left(\mcl T \mbf e(t) \right) &= (\mcl A+\mcl L\mcl C_2) \mbf e(t),\notag\\
% e_z(t) &= \mcl C_1 \mbf e(t), \quad \mcl T \mbf e(0)= -\left(\mcl B_1+\mcl L D_{21} \right) x_0.\label{eq:error_auxiliary}
% \end{align}
%as
\begin{equation}
\min_{\mcl L \in \PI_4}\norm{G(\mcl T, (\mcl A+\mcl L\mcl C_2),-\left(\mcl B_1+\mcl L D_{21} \right) ,\mcl C_1)}_{H_2},\label{prob:estimator}
\end{equation}
using Def.~\ref{def:h2norm}. 
% as
% \begin{equation}
% \min_{\mcl L \in \PI_4}\sup_{\substack{z,\mbf e\, \text{satisfy~\eqref{eq:error_auxiliary}}\\ \norm{x_0}=1}} \norm{e_z}_{L_2}.\label{prob:estimator}
% \end{equation}
%In 
Sec.~\ref{sec:h2estimator} will reformulate the $H_2$-optimal estimation problem as an LPI. First, however, we need to address the problem of computing the $H_2$-norm of a PIE using LPIs. %how to use LPIs to compute the $H_2$ norm of a PIE.
\vspace{-3 mm}

\section{An LPI for the $H_2$ norm}\label{sec:h2norm}\vspace{-1 mm}

In this section, we show how to use LPIs to compute the $H_2$ norm of a PIE. We begin by 
reformulating the following result from~\cite{danilo}.\vspace{-2 mm}

\begin{thm}[~\cite{danilo}]\label{h2-norm}
		%Suppose
        Given $\mcl{T,A},\mcl B_1, \mcl C_1\in \PI_4$, %. If 
        suppose there exist constant $\epsilon>0$ and $\mcl P%=\mcl P^*
        \in \PI_4$ such that $\mcl P\succcurlyeq \epsilon I$, %and %such that:
		\begin{align*}
		\hspace{-3 mm}% \nonumber\\
			\mcl A^{\ast} \mcl P \mcl T + \mcl T^{\ast} \mcl P \mcl A + \mcl C_1^{\ast} \mcl C_1 \preccurlyeq -\epsilon I,
		\end{align*}
and $\trace(\mcl B_1^{\ast} \mcl P \mcl B_1) \leq \gamma^2.$ Then, $\norm{G(\mcl T, \mcl A,\mcl B_1,\mcl C_1)}_{H_2}\leq \gamma.$

\end{thm}\vspace{-2 mm}

We can now use an extension of the Schur complement to obtain an LPI for bounding the $H_2$ norm, which will be used for estimator design in Sec.~\ref{sec:h2estimator}. This reformulation, however, requires us to define vertical and horizontal concatenation of $\PI_4$ operators such that the concatenated operator is in $\PI_4$ (See Lem.~39 and Lem.~40 from~\cite{shivakumarGPDE}). This definition separately concatenates the real and distributed portions of the operator so that if, e.g., $\mcl P \in \mcl L(\R L_2^{n,m})$ and  $\mcl Q \in \mcl L(\R L_2^{p,q})$, then
\[
\bmat{\mcl P&0\\0&\mcl Q}\in \mcl L(\R^{n+p} \times L_2^{m+q}).
\]
In the proof of the following lemma, we do not reorder rows and columns. However, the result holds for the standard definition of concatenation since inequalities are preserved under symmetric reordering of rows and columns.\vspace{-2 mm}

\begin{lem}[Schur Complement]\label{lem:schur}
Suppose $\mcl P,\mcl Q,\mcl R \in \PI_4$. Then the following are equivalent.
\begin{enumerate}
\item
$\displaystyle \bmat{\mcl P&\mcl Q\\\mcl Q^*&\mcl R}\succcurlyeq \epsilon I $ for some $\epsilon>0$.
\item $\displaystyle \mcl R-\mcl Q^* \mcl P^{-1} \mcl Q \succcurlyeq \epsilon I $ and  $\mcl P\succcurlyeq \epsilon I$ for some $\epsilon>0$.
\end{enumerate}

\end{lem}
\begin{proof} In this proof, there is no rearrangement of rows or columns. %Now, 
Mirroring the standard proof of the Schur complement, suppose that 1) is true. Then, we have
\[
\ip{\mbf x}{\mcl P\mbf x}=\ip{\bmat{\mbf x\\0}}{\bmat{\mcl P&\mcl Q\\\mcl Q^*&\mcl R}\bmat{\mbf x\\0}}\ge \epsilon \norm{\mbf x}^2,
\]
which implies that $\mcl P$ is invertible. Now note that
\small{
\[
\bmat{\mcl P&0\\0&\mcl R-\mcl Q^* \mcl P^{-1} \mcl Q}=\bmat{I &\hspace{-2mm}-\mcl P^{-1}\mcl Q\\0&\hspace{-1mm}I}^* \bmat{\mcl P&\mcl Q\\\mcl Q^*&\mcl R}\bmat{I &\hspace{-2mm}-\mcl P^{-1}\mcl Q\\0&\hspace{-1mm}I},
%&\succcurlyeq \epsilon \bmat{I &-\mcl P^{-1}\mcl Q\\0&I}^* \bmat{I &-\mcl P^{-1}\mcl Q\\0&I}\\
%&=\epsilon \bmat{I &0\\-\mcl Q^*\mcl P^{-1}&I} \bmat{I &-\mcl P^{-1}\mcl Q\\0&I}\\
%&=\epsilon \bmat{I &0\\-\mcl Q^*\mcl P^{-1}&I} \bmat{I &-\mcl P^{-1}\mcl Q\\0&I}
\]}
and hence
\begin{align*}
&\ip{\mbf x}{(\mcl R-\mcl Q^* \mcl P^{-1} \mcl Q)\mbf x}=\ip{\bmat{0\\\mbf x}}{\bmat{\mcl P&0\\0&\mcl R-\mcl Q^* \mcl P^{-1} \mcl Q}\bmat{0\\\mbf x}}\\
&=\ip{\bmat{-\mcl P^{-1}\mcl Q \mbf x\\ \mbf x}}{\bmat{\mcl P&\mcl Q\\\mcl Q^*&\mcl R}\bmat{-\mcl P^{-1}\mcl Q \mbf x\\ \mbf x}}\\
&\ge \epsilon \norm{\bmat{-\mcl P^{-1}\mcl Q \mbf x\\ \mbf x}}^2 \ge \epsilon \norm{\mbf x}^2.
\end{align*}

For the converse, suppose 2) is true. Then
\[
\bmat{\mcl P&\mcl Q\\\mcl Q^*&\mcl R}=\bmat{I &\mcl P^{-1}\mcl Q\\0&I}^*\bmat{\mcl P&0\\0&\mcl R-\mcl Q^* \mcl P^{-1} \mcl Q}\bmat{I &\mcl P^{-1}\mcl Q\\0&I},
%&\succcurlyeq \epsilon \bmat{I &-\mcl P^{-1}\mcl Q\\0&I}^* \bmat{I &-\mcl P^{-1}\mcl Q\\0&I}\\
%&=\epsilon \bmat{I &0\\-\mcl Q^*\mcl P^{-1}&I} \bmat{I &-\mcl P^{-1}\mcl Q\\0&I}\\
%&=\epsilon \bmat{I &0\\-\mcl Q^*\mcl P^{-1}&I} \bmat{I &-\mcl P^{-1}\mcl Q\\0&I}
\]
which implies
\[
\ip{\bmat{\mbf x \\\mbf y}}{\bmat{\mcl P&\mcl Q\\\mcl Q^*&\mcl R} \bmat{\mbf x \\\mbf y}}\ge \epsilon \norm{\bmat{I &\mcl P^{-1}\mcl Q\\0&I}\bmat{\mbf x \\\mbf y}}^2.
\]

Now, define $\norm{\bmat{I &\mcl P^{-1}\mcl Q\\0&I}^{-1}}_{\mcl L(\R L_2)}=\delta$. Then
\[
\norm{\bmat{I &\mcl P^{-1}\mcl Q\\0&I}\bmat{\mbf x \\\mbf y}}^2\ge \delta\norm{\bmat{\mbf x \\\mbf y}}^2,
\]
and hence
\[
\ip{\bmat{\mbf x \\\mbf y}}{\bmat{\mcl P&\mcl Q\\\mcl Q^*&\mcl R} \bmat{\mbf x \\\mbf y}}\ge \epsilon \delta \norm{\bmat{\mbf x \\\mbf y}}^2.
\]
%as desired.
\end{proof}\vspace{-6 mm}

\begin{thm}\label{h2-norm_v2}
		Given $\mcl{T,A,B}_1, \mcl C_1 \in \PI_4$, suppose there exist constant $\epsilon>0$, matrix $W$, and $\mcl P \in \PI_4$ such that %a 4-PI operator 
        $\mcl P\succcurlyeq \epsilon I$,% and %such that:
		\begin{align}\label{H2_ineq}
	\hspace{-3 mm}	\bmat{-\gamma I & \mcl C_1  \\ \mcl C_1^* & \mcl T^{\ast} \mcl P \mcl A+\mcl A^{\ast} \mcl P \mcl T }	\preccurlyeq -\epsilon I,%\label{H2_ineq1}\\
\bmat{W&\mcl B_1^*\mcl P\\\mcl P\mcl B_1 &\mcl P}\succcurlyeq \epsilon I,
\end{align}
%\label{H2_ineq2}\\
%\bmat{\mcl P&-(\mcl P\mcl B_1+\mcl ZD_{21})\\-(\mcl P\mcl B_1+\mcl ZD_{21})^*& W}&\succcurlyeq 0\\
and	$\trace( W) \leq \gamma.$ 
		%\end{align}
Then, $\norm{G(\mcl T, \mcl A,\mcl B_1,\mcl C_1)}_{H_2}\leq \gamma.$
% $ \displaystyle
% \sup_{\substack{z,\mbf x\, \text{satisfy~\eqref{eqn:PIEprime}}\\ \norm{x_0}=1}} \norm{z}_{L_2}\leq \gamma.
% $
\end{thm}\vspace{-4 mm}

\begin{proof}
    Suppose $\gamma, \mcl P, \mcl Z$ are as stated above. Then, the first inequality in %Inequality
    Eq.~\eqref{H2_ineq} combined with Lem.~\ref{lem:schur} implies\vspace{-4 mm}
    
    \[
    \mcl A^{\ast} \mcl P \mcl T + \mcl T^{\ast} \mcl P \mcl A + \gamma^{-1} \mcl C^{\ast} \mcl C \preccurlyeq -\epsilon I.
    \]\vspace{-6 mm}
    
    Likewise, the second inequality in Eq.%Inequality
    ~\eqref{H2_ineq} %combined with Lem.~\ref{lem:schur} 
    implies
    \begin{align*}
         W-\mcl B^*\mcl P\mcl P^{-1}\mcl P\mcl B=W-\mcl B^*\mcl P\mcl B  \succ 0.
            \end{align*}
            Now $W$ and $\mcl B^*\mcl P\mcl B$ are matrices and hence $\trace (\mcl B^*\mcl P\mcl B)< \trace W \leq \gamma$. Define $\hat{\mcl{P}} = \gamma \mcl P$ so that $\mcl P = \gamma^{-1}\hat{\mcl{P}} $ and hence
            \[
  \mcl A^{\ast} \hat{\mcl P} \mcl T + \mcl T^{\ast} \hat{\mcl P} \mcl A + \mcl C^{\ast} \mcl C \preccurlyeq -\gamma \epsilon I,           \qquad \trace (\mcl B^*\hat{\mcl P}\mcl B) \leq \gamma^2,
            \]
            which implies the conditions of Thm.~\ref{h2-norm} are satisfied.
\end{proof}\vspace{-2mm}

In the next section, Thm.~\ref{h2-norm_v2} is used to design observers that minimize a bound on the $H_2$ norm of the error dynamics.\vspace{-2mm}
			
\section{An LPI for $H_2$-optimal Estimation}\label{sec:h2estimator}\vspace{-2 mm}

%In t
This section, %we 
considers the problem of designing the estimator gain $\mcl L \in \PI_4$ which minimizes a bound on the $H_2$ norm of the error dynamics defined in Subsec.~\ref{subsec:problem_est}.\vspace{-1 mm} %Specifically, recall these error dynamics are given by

\begin{thm}\label{thm:h2-estimator}
Given $\mcl{T,A,B}_1, \mcl C_1, \mcl C_2 \in \PI_4$, $D_{21} \in \R^{n_y \times n_w}$, suppose there exist constant $\epsilon>0$, matrix $W$, and $\mcl{P,Z} \in \PI_4$ %Suppose there exist $\epsilon>0$, matrix $W$, and %4-PI operators 
%$\mcl P, \mcl Z  \in \PI_4$, 
such that $\mcl P%=\mcl P^* 
\succcurlyeq \epsilon I$, %and $\mcl Z$, such that
\begin{align*}
		\bmat{-\gamma I & \mcl C_1  \\ \mcl C_1^* & \mcl T^{\ast} \mcl P \mcl A+\mcl A^{\ast} \mcl P \mcl T +\mcl T^*\mcl Z\mcl C_2+\mcl C_2^*\mcl Z^*\mcl T}	&\preccurlyeq -\epsilon I,\\
\bmat{W& -(\mcl B_1^*\mcl P+D_{21}^T\mcl Z^*)\\-(\mcl P\mcl B_1+\mcl ZD_{21}) &\mcl P}&\succcurlyeq \epsilon I, \end{align*}%\\
%\bmat{\mcl P&-(\mcl P\mcl B_1+\mcl ZD_{21})\\-(\mcl P\mcl B_1+\mcl ZD_{21})^*& W}&\succcurlyeq 0\\
and $\trace( W) \leq \gamma.$ 
%\end{align*}
Then, %if $\mcl L = \mcl P^{-1}\mcl Z$,
\[\norm{G(\mcl T, (\mcl A+\mcl L\mcl C_2),-\left(\mcl B_1+\mcl L D_{21} \right) ,\mcl C_1)}_{H_2} \leq \gamma,\] where $\mcl L = \mcl P^{-1}\mcl Z$.%the $H_2$-norm of the system in Eq.~\eqref{eq:PIEerror} is upper bounded by $\gamma$. 
\end{thm}
\begin{proof}
Let $\mcl L = \mcl P^{-1}\mcl Z$. Then
 		\begin{align*}
&		\bmat{-\gamma I & \mcl C_1  \\ \mcl C_1^* & \mcl T^{\ast} \mcl P \left(\mcl A +\mcl L \mcl C_2\right)+\left(\mcl A +\mcl L \mcl C_2\right)^{\ast} \mcl P \mcl T }	\\
&=\bmat{-\gamma I & \mcl C_1  \\ \mcl C_1^* & \mcl T^{\ast} \mcl P \left(\mcl A +\mcl P^{-1}\mcl Z \mcl C_2\right)+\left(\mcl A +\mcl P^{-1}\mcl Z \mcl C_2\right)^{\ast} \mcl P \mcl T }\\
&=\bmat{-\gamma I & \mcl C_1  \\ \mcl C_1^* & \mcl{ T^{\ast} P} \mcl A+\mcl{ A^{\ast} P} \mcl T +\mcl T^*\mcl Z\mcl C_2+\mcl C_2^*\mcl Z^*\mcl T}	\preccurlyeq -\epsilon I,
		\end{align*}
and
 		\begin{align*}
&\bmat{W&-\left(\mcl B_1+\mcl L D_{21}\right)^*\mcl P\\-\mcl P\left(\mcl B_1+\mcl L D_{21}\right) &\mcl P}\\
&=\bmat{W& -(\mcl B_1^*\mcl P+D_{21}^T\mcl Z^*)\\-(\mcl P\mcl B_1+\mcl ZD_{21}) &\mcl P}\succcurlyeq \epsilon I.
		\end{align*}
Application of Thm.~\ref{h2-norm_v2} completes the proof.
%Since $\trace(W) \le \gamma$, from Thm.~\ref{h2-norm_v2}, we conclude that $\gamma$ is an upper bound on the $H_2$-norm of the %PIE 
%system %defined by $\{\mcl T,(\mcl A+\mcl L\mcl C_2),-(\mcl B_1+\mcl LD_{21}),\mcl C_1 \}$ as 
%in Eq.~\ref{eq:PIEerror}.
\end{proof}%\vspace{-3 mm}

\section{Estimator Gain Reconstruction}\label{sec:reconstruction}\vspace{-1mm}

In this section, we 
suppose that $\mcl P,\mcl Z$ minimize $\gamma$, subject to the constraints in Thm.~\ref{thm:h2-estimator}, and
 %Recall that, as discussed in Sec.~\ref{subsec:LPIs}, that the LPI constraints of Thm.~\ref{thm:h2-estimator} can be posed as LMIs. Consequently, %the problem of finding optimal $\mcl P, \mcl Z$ that minimize $\gamma$ is solved by solving the corresponding LMI problem. 
%the optimization problem is solved by solving the corresponding LMI
%
 construct the observer gain $\mcl L = \mcl P^{-1}\mcl Z$.
 %and implement %these gains 
%$\mcl L$ numerically to simulate 
%the PIE estimator of Eq.~\eqref{PIEobserver2} numerically to 
%track the state of a PDE. 
%
%The parsing and solution process can be automated by the Matlab toolbox PIETOOLs, which is used for the numerical experiments of Sec.~\ref{sec:num}. 
%
First, %we 
note that if $\mcl P\in \PI_4$ is invertible% = \fourpi{P}{Q}{Q^T}{R_i}$
, then the inverse $\mcl P^{-1}$ can be computed using, e.g. Lem.~17 in~\cite{shivakumardual} and numerically approximated by a PI operator
\[
\mcl P^{-1} \approx \hat{\mcl P}:= \fourpi{\hat P}{\hat Q}{\hat Q^T}{\hat R_i}.
\]
Furthermore, for $\mcl Z= \fourpi{Z_1}{\emptyset}{Z_2}{\emptyset}$, we have, by the 4-PI composition formula~\cite{shivakumarGPDE}, that $\mcl L= \fourpi{L_1}{\emptyset}{L_2}{\emptyset}$, where\vspace{-4 mm}

\begin{align*}
L_1&=\hat P Z_1 + \int_a^b \hat Q(s)Z_2(s) ds, \, \\
L_2(s)&=\hat Q(s)^T Z_1+\hat R_0(s)Z_2(s)
 +\int_a^b \hat R(s,\theta)Z_2(\theta)d \theta, %+\int_s^b \hat R_2(s,\theta)Z_2(\theta)d \theta,
\end{align*}
$\hat{R}(s,\theta)=\hat{R}_1(s,\theta)$ for $a \leq \theta \leq s$, $\hat{R}(s,\theta)=\hat{R}_2(s,\theta)$ for $s < \theta \leq b$. $L_1$ represents the correction to the ODE state, and $L_2(s)$ represents a correction to the distributed state. % In the following section, %we test 
%observers designed in this manner are tested by numerical integration. % of a PIE estimator using the output from the numerical integration of the PDE.\vspace{-2mm} %being observed.\vspace{-3mm}%it is observing.
%\vspace{- 1mm}

\section{Numerical Examples}\label{sec:num}\vspace{-1mm}

In this section, we validate the proposed algorithm for observer synthesis by constructing the $H_2$-optimal observer gains and numerically integrating the estimator dynamics using the output from numerical integration of the associated %PDEs 
system subject to disturbances. Illustrative examples include a delay system (Ex.~A), and two PDEs: an unstable non-homogeneous reaction-diffusion equation (Ex.~B) and an energy-preserving Euler-Bernoulli beam equation (Ex.~C). For the delay system, the resulting estimator is compared with the non-convex approach to $H_2$ estimation taken in~\cite{suh2006h}. 

The conversion to PIE %is included for DDE case while for the PDEs,
is automated by the command-line %PDE
input option of PIETOOLS~\cite{manual}. %is used to obtain the PIE representation.
Solution of the LPI in Thm.~\ref{thm:h2-estimator}, operator inversion, and estimator gain reconstruction are likewise performed using PIETOOLS. Numerical integration of both the PIE estimator and PDE plant is performed using a Galerkin projection with Chebyshev bases of order up to $8$, %and as 
implemented in PIESIM~\cite{peet2020piesim}. %In each case, we plot both the evolution of the performance metric being minimized ($e_z$) as well as the error in the estimate of the distributed PDE state.

\noindent \textbf{Example A: } %\label{ex:tds}
 Consider the time-delay system from~\cite{suh2006h} %from~\cite{suh2006h}:
\begin{align*}
    \dot{ x}(t) &= A_0  x(t) + A_d  x(t-\tau) +B_{1}w(t)+ B_{2} u(t),\\
     y(t)&=C_{2}x(t)+C_{d}x(t-\tau)+D_{21}w(t), \,  z(t) = C_{1}  x(t), \notag
\end{align*}
where $x(t), z(t), w(t) \in \R^{2}$, $u(t), y(t) \in \R$, and
\begin{align*}
    &A_0 = \bmat{
        -2 & 1\\
        0 & -1 
    },\,  A_d = \bmat{
        -1 & 0 \\
        -1 & -1 }, \,C_{1} = \text{I}_{2}, \,C_{2} = \bmat{0 \\1 }\\
    & B_{1} =  \bmat{0.2& 0\\0.2& 0}
    , \,B_{2} = C_d=\bmat{1 \\1}, \,%C_{d} = \bmat{1 \\1 }, 
    D_{21}=\bmat{0 &0.5}.
\end{align*}
 %is equivalent to the following PDE, as shown in~\cite{peet2021representation}, %, this system is equivalent to the PDE %representation
% \begin{align*}
%     \dot x(t) &= A_0 x(t) +  A_d \phi(t,-1) + B_{1} w(t)+B_{2}u(t),\\ 
%     y(t) &=C_2 x(t) +C_d \phi(t,-1)+D_{21}w(t),  \; t\geq 0, \\
%     \partial_t \phi(s,t) &=I_{\tau} \partial_s \phi(s,t),\; \forall s \in [-1,0), \text{and}\;  \phi(0,t)=x(t),  
% \end{align*}
% where $\phi(s,t) :=x(t+\tau s)$ is assumed in the Sobolev space of $L_2$-differentiable, $\R^2$-valued functions, and $I_{\tau}=1/\tau \text{I}_{2}$.
Following the construction in~\cite{peet2021representation}, the equivalent PIE representation, % Eq.~\eqref{PIE}, 
is obtained by defining the PIE state %$\mbf x(t)=\partial_s \phi(\cdot,t)$ and%of this PDE,
%as detailed in~\cite{peet2021representation}, is given by
$x(t)=\tau \dot x(t+\tau s)$, which yields
\begin{align*}
\mcl T &=\fourpi{\text{I}_{2}}{0}{\text{I}_{2}}{0,0,-\text{I}_{2}},&\mcl B_1 &=\fourpi{B_{1}}{\emptyset}{0}{\emptyset},\\
\mcl A &=\fourpi{A_{0}+A_{d}}{-A_{d}}{0}{I_{\tau},0,0},& \mcl B_2 &=\fourpi{B_{2}}{\emptyset}{0}{\emptyset},\\ \mcl C_2 &=\fourpi{C_{2}+C_{d}}{-C_{d}}{\emptyset}{\emptyset} , & \mcl C_1 &=\fourpi{C_{1}}{0}{\emptyset}{\emptyset}.
\end{align*}

To show the performance of the $H_2$-optimal estimator resulting from Thm.~\ref{thm:h2-estimator}, we simulate the system with a time step of $0.001 s$. %In Fig.~\ref{fig:observer-dde}(a), four steady state responses are shown for zero initial conditions and $u(t)=0$. In (I) and (III), a sinusoidal process disturbance $w_1(t)=d_{\omega}(t):=sin(\omega t), \, t \geq 0$ is used with $w_2(t)=0$. In (II) and (IV), we use sensor noise $w_2(t)=d_{\omega}(t)$ with $w_1(t)=0$. It can be seen that the gain from disturbance to $\mcl C_1 \tilde{x}(t) = [\tilde x_1(t) ~\,~ \tilde x_2(t)]^T$ is small when the input is $\nu(t)$, as well as when the input is $\mu(t)$, for the two frequencies $\omega$ considered. %, considering sinusoidal noises of two different frequencies $\omega$. %(note that $\mcl C_1\tilde{x}(t)=[x_1(t) ~\,~ x_2(t)]^T$). 
% The magnitude of the %steady state 
% frequency responses are approximately equal for both states, and can be computed: %$~|\hat G_{\nu}(10 \imath)|=-33.97%794 
% %dB$, $|\hat G_{\mu}(10\imath)|=-47.96%588 
% %dB$, $~|\hat G_{\nu}(100 \imath)|=-106.02%06
% %dB$, and $|\hat G_{\mu}(100\imath)|=-67.96%588 
% %dB$; where $\hat G_{\nu}$ relates to $\nu=sin(\omega t), t \geq 0$ and $\mu=0$, whereas $\hat G_{\mu}$ corresponds to $\nu=0$ and $\mu=sin(\omega t), t \geq 0$.  
% \[
% \bmat{\left|\hat G_{\nu}(10\imath)\right| & \left|\hat G_{\mu}(10\imath)\right|\\\left|\hat G_{\nu}(100\imath)\right| & \left|\hat G_{\mu}(100\imath)\right|} \approx\bmat{-33.97 & -47.96\\-106.02 & -67.96} dB,
% \]
%  up to two decimal places, where $\hat G_{\nu}(\omega\imath)$ is the frequency response to input $\nu$, and $\hat G_{\mu}(\omega\imath)$ corresponds to input $\mu$.  
Fig.~\ref{fig:observer-dde} illustrates a numerical simulation for: non-zero initial conditions; a unit step $u(t)$; and %sinusoidal 
concatenated process and measurement noises $w(t)=e^{-t}[sin(10t)~\,~ sin(100t)]^T$, demonstrating convergence of the error to zero when the inputs are $L_2$-bounded. Tab.~\ref{table_norms} shows that the optimal bounds % on the closed-loop $H_{2}$-norm 
obtained using Thm.~\ref{thm:h2-estimator} are consistent with those obtained in~\cite{suh2006h}.%consistently less conservative than 
 %The $H_2$ norm computed numerically by integrating the output of the error dynamics Eq.~\eqref{PIEobserver2}, is $0.0066$ for $x_0=[1 ~\,~0]^T$.%the bounds obtained by~\cite{suh2006h}.%, for different values of the delay.
%\vspace{-3 mm}

%\end{ex}%}
% \begin{figure}[t]
%     \centering
%        \includegraphics[scale=0.3,keepaspectratio]{Figures/H2DDEEX.pdf}
%     \caption{Estimation of system state (Eq.~\eqref{eq:DDE}) using an $H_2$-optimal estimator. System is simulated with delay $\tau = 0.3 s$, with the control input $u(t)=1$ for all $t \geq0$, disturbances $w(t)=\bmat{sin(10t)&sin(100t)}^T$, for $t\geq 0$, and initial conditions $x_{1}(t)=1, x_2(t)=0$, for $t \in [-0.3,0]$. Dashed lines show the estimator state, whereas the solid lines show the system state.} %(a): Evolution of error in the state. (b): Evolution of the regulated output ($z(t)$) of both estimator and DDE.}
%     \label{fig:observer-dde}
% \end{figure}
\begin{figure}[b]
    \centering
    % \begin{subfigure}{0.49\linewidth}
    %    \includegraphics[width=\linewidth, height=40mm]{Figures/H2DDEEX_Bode.pdf}
    % \caption{}
    % \end{subfigure}
   % \hfill
    %\begin{subfigure}{0.49\linewidth}
    \includegraphics[width=\linewidth,height=40mm, keepaspectratio]{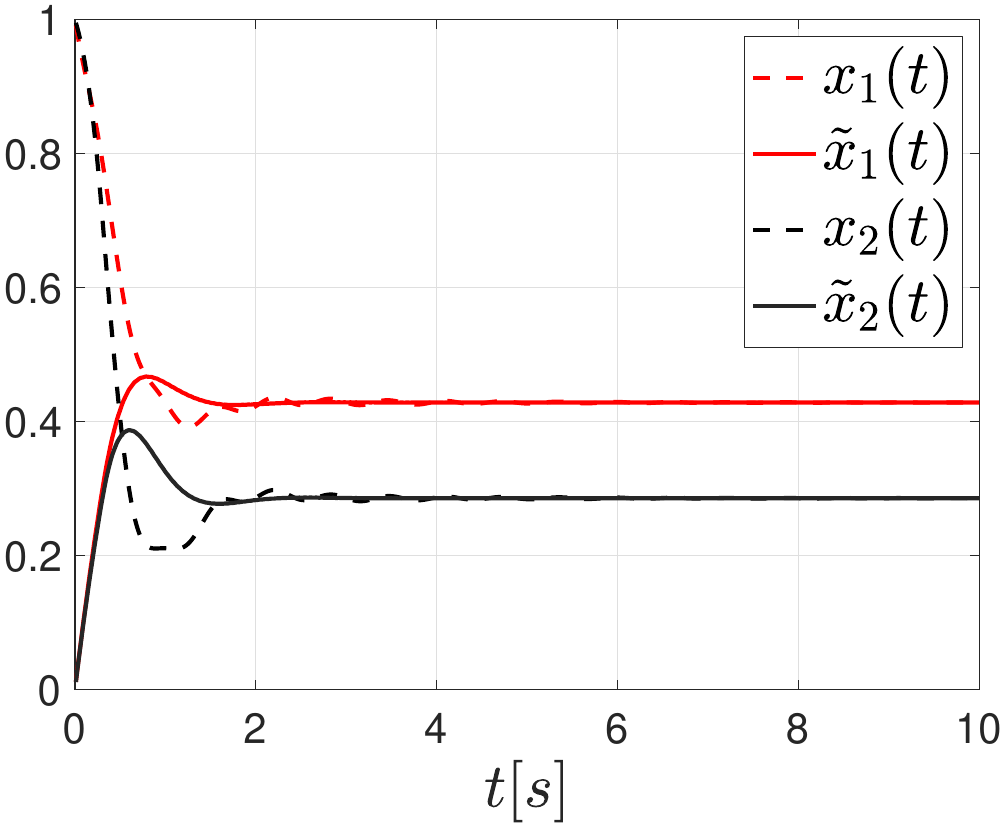}
   %        \caption{}
   % \end{subfigure}
     \caption{Numerical estimation of an $H_2$-optimal estimator for a time-delay system (Ex.~A), with $\tau = 0.3 s$, along with process and sensor noise. The estimated states are simulated with %the sinusoidal disturbance $d_{\omega}(t)=sin(\omega t), \, t \geq 0$ in (a).(I) corresponds to $\nu(t)=d_{10}(t)$ and $\mu=0$; (II) corresponds to $\mu(t)=d_{10}(t)$, and $\nu=0$; (III) corresponds to $\nu(t)=d_{100}(t)$ and $\mu=0$; and, in (IV), $\mu(t)=d_{100}(t)$, and $\nu=0$.
     %Solid red lines correspond to $\tilde x_1(t)$, and solid black lines to $\tilde x_2(t)$. In (b), the states are shown for 
     the input $u(t)=1, \, t \geq0$; the initial conditions $(x_{1}(t),x_2(t))=(1,1)$ for $t \in [-0.3,0]$; %and %the disturbance 
     %$w(t)=e^{-\frac{t}{2}}\bmat{sin(10t)&sin(100t)}^T,\, t \geq0$
     process noise $w_1(t)=e^{-\frac{t}{2}}sin(10t), \, t\geq 0$, and measurement noise $w_2(t)=e^{-\frac{t}{2}}sin(100t), \,t\geq 0$. Red lines correspond to $\tilde x_1(t)$, and black lines correspond to $\tilde x_2(t)$. %, $t\geq 0$.
     Dashed lines show the system states, and solid lines show the estimated states.
     }
     \label{fig:observer-dde}
    \end{figure}
%\begin{ex}%[Unstable Reaction-Diffusion Equation]
%\label{ex:react-diff}
\noindent \textbf{Example B: }
In this example, consider the unstable %, non-homogeneous 
reaction-diffusion PDE with sensor and process noise.%, where sensor measurements are taken at the boundary.
%\vspace{-3 mm}
\begin{align*}%\label{eq:react-diff}
    &\dot{\mbf \xi}(t,s) = 3\mbf \xi(t,s)+(s^2+0.2)\partial_s^2\mbf \xi(t,s)-\frac{s^2}{2}w(t),\notag\\
    &\mbf \xi(t,0)=\partial_s \mbf \xi(t,1)=0, \quad y(t) = \mbf \xi(t,1)+w(t),
\end{align*}
and $z(t)= \int_0^1 \mbf \xi(t,\theta)d\theta$. PIETOOLS is used to obtain the PIE representation of this PDE with PIE state $\mbf x(t)=\partial_s^2\xi(t)$:%, and system parameters
\begin{align*}
\mcl T &=\fourpi{\emptyset}{\emptyset}{\emptyset}{0,R_1,R_2}, &\mcl B_1 &=\fourpi{\emptyset}{\emptyset}{0.5s^2}{\emptyset},\\
\mcl A &=\fourpi{\emptyset}{\emptyset}{\emptyset}{S_0,3R_1,3R_2}, & \mcl C_1 &=\fourpi{\emptyset}{p(s)}{\emptyset}{\emptyset},\\ \mcl C_2 &=\fourpi{\emptyset}{-s}{\emptyset}{\emptyset} , \quad D_{21} =1, &p(s)&=\frac{s^2-2s}{2}
\end{align*}
where $R_1(s,\theta)=-\theta$, $R_2(s,\theta)=-s$, and $S_0(s)=s^2+0.2$.
% \begin{align*}
%     R_0(s)&=0,\quad R_1(s,\theta)=-\theta, \quad R_2(s,\theta)=-s,\\
%     S_0(s)&=s^2+0.2, \quad S_1(s,\theta)=-2\theta, \quad S_2(s,\theta)=-3s.
% \end{align*}
We numerically simulate the PDE and $H_2$-optimal estimator with a time step of $0.002 s$, $w(t)=e^{-t}sin(100t)$, and PDE initial condition $\mbf \xi(0,s)=s^2/2-s$ implying $\mbf x(0,s) = 1$.
As seen in Fig.~\ref{fig:observer-react-diff}, the estimation errors for both the state and the regulated output decay quickly despite instability in the PDE and high-frequency excitation.%\vspace{-2 mm}
%\end{ex}
%A time step of $0.002 s$ is used in the simulation.
%\subsection{Euler-Bernoulli Beam Equation}

\noindent \textbf{Example C: }
%\begin{ex}\label{ex:EB}%[Euler--Bernoulli Beam Equation]
Consider a cantilevered Euler-Bernoulli beam, with displacement $\eta(t,s)$, with both sensor and process noise, where the sensor measures tip velocity at the right boundary. The PDE may be written in the first-order form
% \begin{align}\label{eq:EBSS}
% 	&\ddot{\eta}(t,s) = -\frac{1}{10}\partial_s^4 \eta(t,s)+\frac{s^2-2s}{2}w(t),\notag\\
% &\eta(t,0) = \partial_s \eta(t,0) = \partial_s^2\eta(t,1) = \partial_s^3 \eta(t,1) = 0,\notag\\
% &z(t)= \int_0^1 \dot{\eta}(t,s) ds,\quad y(t) = \dot{\eta}(t,1)+w(t).
% \end{align}
%We may rewrite this equation in first-order form by defining the concatenated state $\mbf v(t,s)=\left(\dot{ \eta}(t,s), \partial_s^2\eta(t,s)\right)$~\cite{peet2021partial}, 
by defining $\mbf v(t,s)=\left(\dot{ \eta}(t,s), \partial_s^2\eta(t,s)\right)$~\cite{peet2021partial},  
%we obtain the equivalent first-order PDE%to obtain the coupled PDE system
as
\begin{align*}%\label{eq:EBSS}
&\dot{\mbf{v}}(t) = \bmat{0&-0.1\\1&0}\partial_s^2\mbf{v}(t)+\bmat{\frac{s^2-2s}{2}\\0}w(t)+\bmat{1\\0}u(t),\label{eq:EBSS}\\
&\bmat{1&0} \mbf q(t,0) = \bmat{0&1} \mbf q(t,1) =0, \text{where } \mbf q:=\mbf v-\partial_s \mbf v \notag\\%\mbf v(t,0) = \bmat{1&0}\partial_s\mbf v(t,0) = 0\notag\\
%\bmat{0&1}\mbf v(t,1) =\bmat{0&1}\partial_s\mbf v(t,1) = 0, \\
&z(t) = \int_0^1 \bmat{1&0}\mbf v(t,s)ds,\quad y(t) = \bmat{1&0}\mbf v(t,1)+w(t) \notag.
\end{align*}
As in Ex.~B, we find the PIE system parameters to be
\begin{align*}
\mcl T &=\fourpi{\emptyset}{\emptyset}{\emptyset}{R_0, R_1, R_2}, \, \mcl A =\fourpi{\emptyset}{\emptyset}{\emptyset}{S_0, S_1, S_2},\\
\mcl B_1&=\fourpi{\emptyset}{\emptyset}{[p(s)~\,~0]^T}{\emptyset}, \,\mcl C_1=\fourpi{\emptyset}{[p(s)+1/2~\,~0]}{\emptyset}{\emptyset}, \\
\mcl C_2 &=\fourpi{\emptyset}{[1-s~\,~0]}{\emptyset}{\emptyset} , \, D_{21} =1, \quad p(s)=\frac{s^2-2s}{2}
\end{align*}
\begin{align*}
   &R_0(s)=S_1(s,\theta)=S_2(s,\theta)=\bmat{0 &\hspace{-1.5mm}0\\0&\hspace{-1.5mm}0},\; S_0(s)=\bmat{0 &\hspace{-2mm}-0.1\\1&\hspace{-2mm}0},\\ 
   &R_1(s,\theta)=\bmat{s-\theta &0\\0&0},\; R_2(s,\theta)=\bmat{0&0\\0&\theta-s},
\end{align*}
 with PIE state $\mbf x(t)=\partial_s^2\mbf v(t)$. We numerically simulate the Euler-Bernoulli beam and $H_2$-optimal estimator with $w(t)=e^{-\frac{t}{2}}sin(10t)$, PDE initial condition $\mbf v(0,s)=(s^2/2,0)$ ($\mbf x(0,s) = (1,0)$), and a time step of $0.001 s$. Visualizing the estimation errors in Fig. \ref{fig:EBbeam}, we again see that the estimation errors in both the state and the regulated output decay quickly while the energy of the beam itself is preserved.
%\end{ex}
%\vspace{-3 mm}
\section{Conclusion}\label{sec:conclusion}\vspace{-1 mm}

\begin{table}[b]
    \caption{Bounds on the $H_2$-norm of the estimator for system in Ex.~A obtained from~\cite{suh2006h} and Thm.~\ref{thm:h2-estimator}.}	
    \small
    \centering
    \begin{tabular}{ |m{1.9 cm} | m{.8cm} |  m{0.8cm} |  m{.8cm} |   m{0.8cm} | }
        \hline
        
        $\tau$  & 0.1  & 0.3 & 0.5 & 0.7\\ \hline
        
        Suh, et al. \cite{suh2006h} & 0.1342 &  0.1559 & 0.1792 & 0.2059\\ \hline
        
        Thm.~\ref{thm:h2-estimator} &0.1326 & 0.1546 & 0.1771 & 0.2009  						\\ \hline	
    \end{tabular}
    \label{table_norms}
\end{table}
%\vspace{- 1mm}
%\vspace{- 10 mm}
% The $H_2$ norm is a commonly used performance metric in the estimation of linear state-space systems. However, 
Finding observers with optimal $H_2$ norm for a delayed or PDE system is complicated by the lack of an equivalent time-domain characterization of this norm. To address this problem, we have proposed an alternative initial condition to output characterization of the $H_2$ norm and applied this characterization to the PIE representation of the error dynamics. This approach allows the optimal observer synthesis problem to be posed as an LPI, which can then be solved using existing software. The proposed approach allows for efficient design of estimators, with provable bounds on performance, for a large class of PDEs and delay systems. The results were applied to design $H_2$-optimal estimators for 3 examples and validated using simulation.
% The results were applied to the estimation of the distributed state using boundary measurement subject to process and sensor noise and validated using numerical simulation of an unstable non-homogeneous reaction-diffusion equation, an energy-preserving Euler-Bernoulli beam equation, and a time-delay system. Automation of the proposed algorithms in PIETOOLS allows for the efficient design of estimators, with provable bounds on stability and performance, for a large class of speculative and data-based models of PDEs coupled with ODEs. 
\begin{figure}[t]
    \centering
    \begin{subfigure}{0.49\linewidth}
       \includegraphics[width=\linewidth, height=40mm]{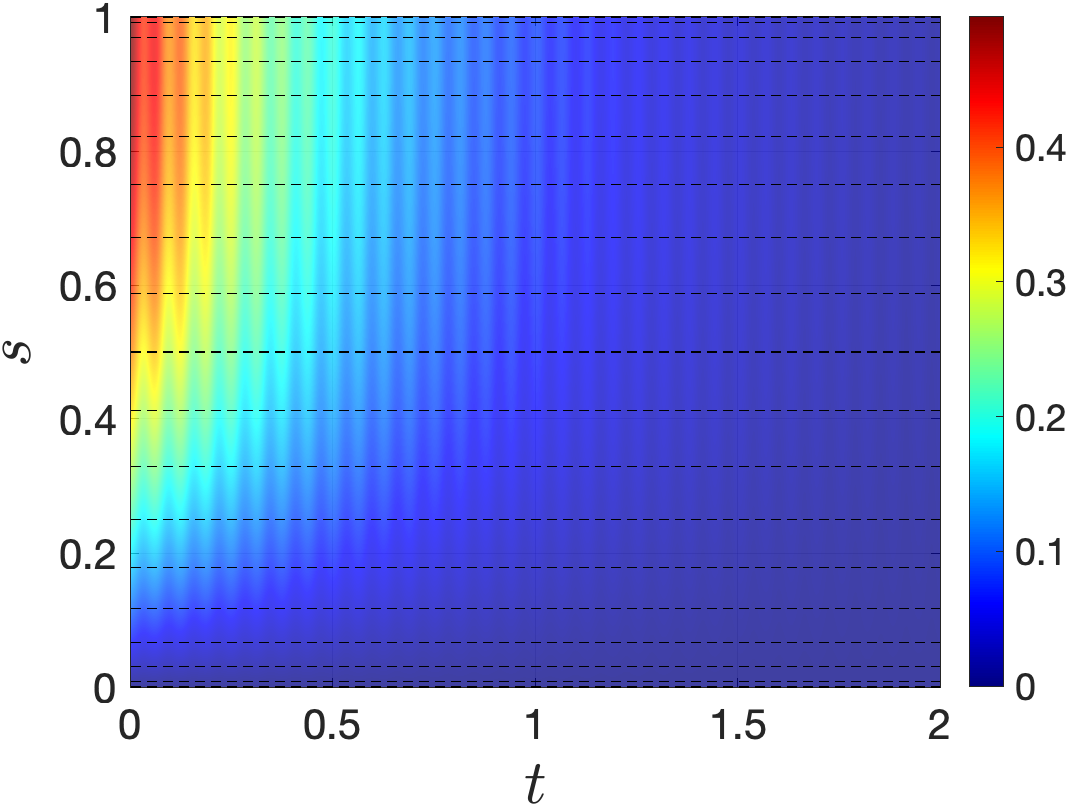}
    \caption{}
    \end{subfigure}
   % \hfill
    \begin{subfigure}{0.49\linewidth}
           \includegraphics[width=\linewidth,height=40mm]{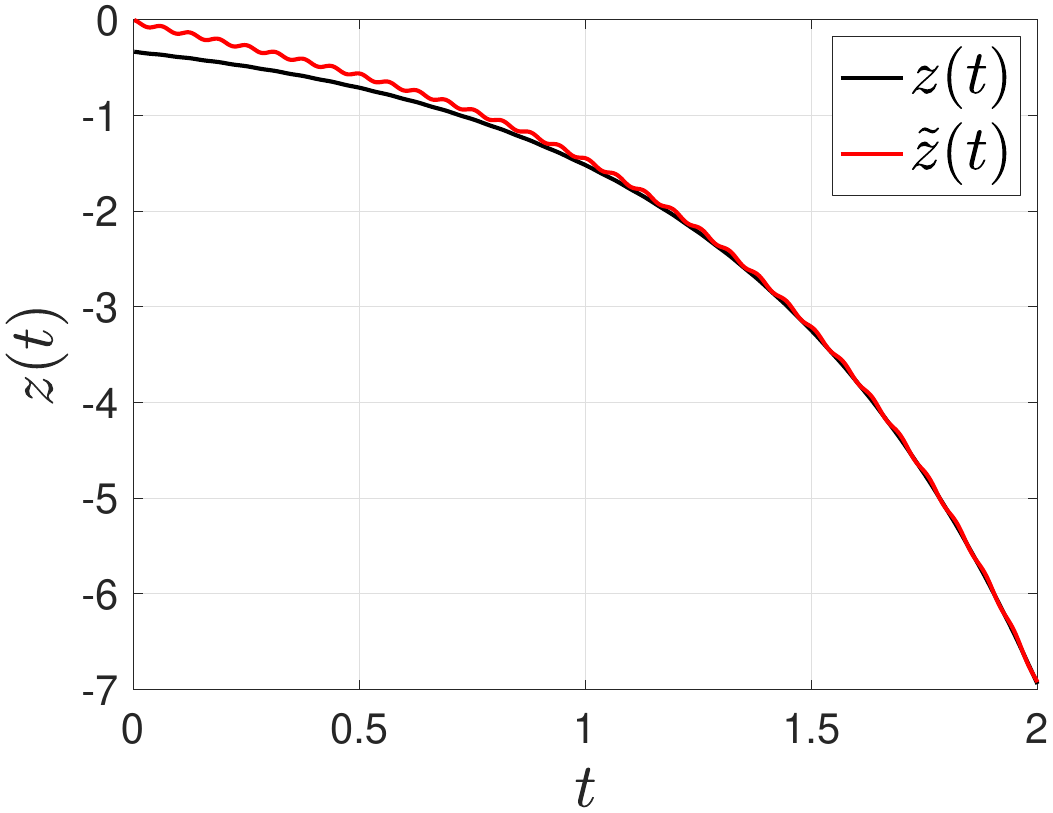}
    \caption{}
    \end{subfigure}
    \caption{Numerical estimation of an $H_2$-optimal estimator for an unstable reaction-diffusion equation (Ex.~B) using measurement at the boundary along with process and sensor disturbance $w(t)=e^{-t}sin(100t)$ and PDE initial condition $\mbf \xi(0,s)=s^2/2-s$. %($\mbf x(0,s) = 1$).
    (a): Evolution of error in estimate of the PDE state $\mcl T \mbf e(t)=\mcl T\tilde{\mbf x}(t)-\mbf \xi(t)$. (b): Evolution of the regulated output $z(t)$ of both estimator and PDE.}
    \label{fig:observer-react-diff}
\end{figure}
     \begin{figure}[t]
    \centering
    \begin{subfigure}{0.49\linewidth}
    \includegraphics[width=\linewidth, height=40mm]{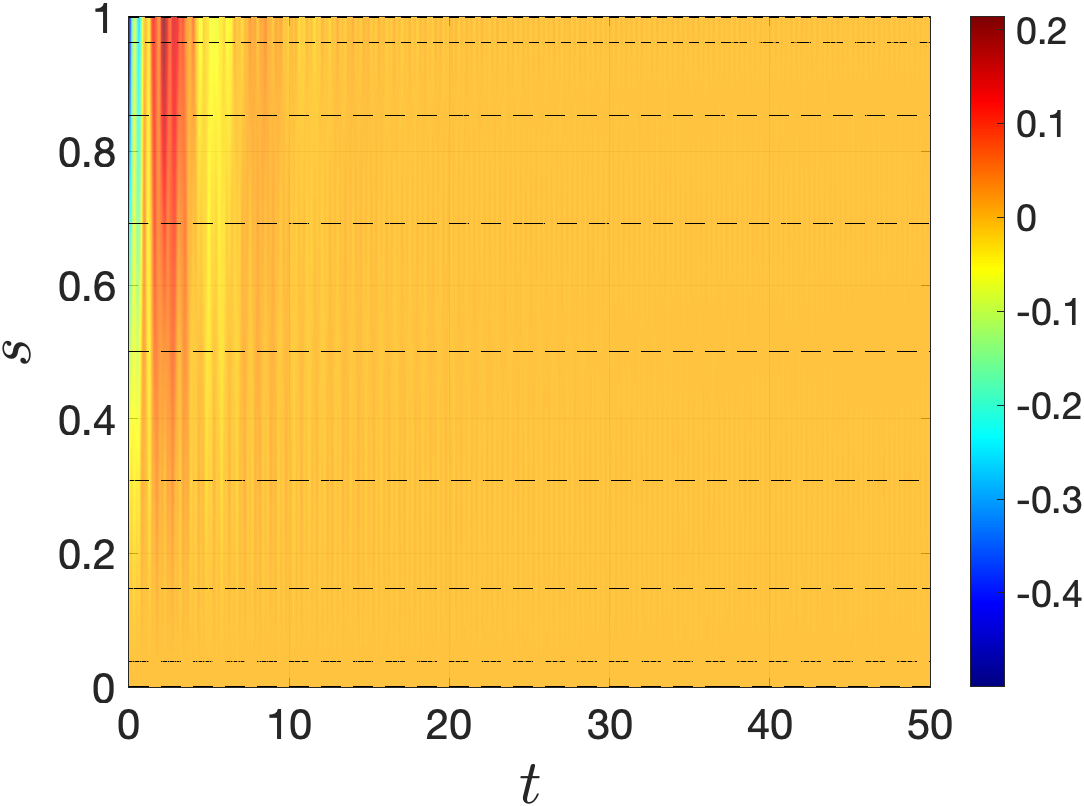}
    \caption{}
    \end{subfigure}
   % \hfill
    \begin{subfigure}{0.49\linewidth} \includegraphics[width=\linewidth,height=40mm]{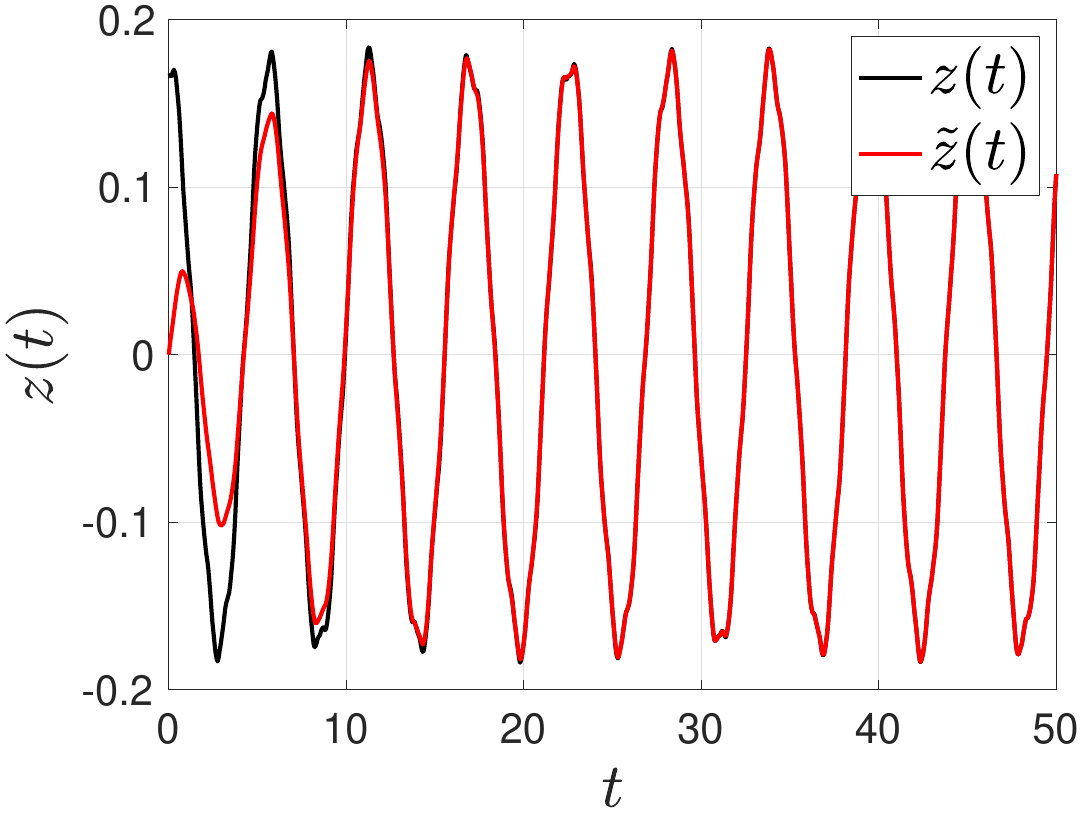}
    \caption{}
    \end{subfigure}
    \caption{Numerical estimation of an $H_2$-optimal estimator for a neutrally stable Euler-Bernoulli beam equation (Ex.~C) using velocity measurement at the boundary with disturbance $w(t)=e^{-\frac{t}{2}}sin(10t)$ and with PDE initial conditions $\dot{\eta}(0,s)=s^2/2$, $\partial_s^2 \eta (0,s)=0$. (a): Evolution of error in estimate of the PDE state: $ \dot{\tilde{\eta}}(t,\cdot)-\dot{\eta}(t,\cdot)$. (b): Evolution of the regulated output ($z(t)$) of both estimator and PDE.}
    \label{fig:EBbeam}
\end{figure}
%\vspace{-1 cm}
\addtolength{\textheight}{-4cm}
\bibliography{Braghini_CDC2025_H2}
\bibliographystyle{IEEEtran}
\end{document}